\documentclass[12pt]{article}
\usepackage{amsmath}
\usepackage{amssymb}
\usepackage{latexsym}
\usepackage{epsf}
\usepackage{supertabular}

\usepackage{rotating}
\usepackage{multicol}
\usepackage{multirow}
\usepackage{epsfig}
\usepackage{fancyhdr}       

\newif\ifger

\gerfalse

\newtheorem{theorem}{Theorem}[section]

\newtheorem{lemma}[theorem]{Lemma}

\newtheorem{remark}[theorem]{Remark}

\newcommand{\cC}{{\mathfrak C}}

\newcommand{\cL}{{\cal L}}

\newcommand{\cP}{{\cal P}}

\newcommand{\cS}{{\cal S}}

\def\la{{\langle}}

\def\PG{{\rm PG}}

\def\Soc{{\rm Soc}}
\def\Aut{{\rm Aut}}

\newcommand{\Fix}{{\rm Fix}}
\newcommand{\Sym}{{\rm Sym}}

\newcommand{\GL}{{\rm GL}}

\newcommand{\AGL}{{\rm AGL}}

\def\la{\lambda}
\def\Fix{{\rm Fix}}

\def\AGL{{\rm AGL}}

\def\Soc{{\rm Soc}}

\title{Classification of line-transitive
point-imprimitive linear spaces with line size at most 12}

\author{Cheryl E. Praeger \\
{\it\small School of Mathematics and Statistics,} \\
{\it \small The University of Western
Australia, Crawley, WA 6009, Australia}\\
{\it\small email: praeger@maths.uwa.edu.au
\medskip }\\
Shenglin Zhou\\
{\it \small Department of Mathematics, Shantou University,}\\
{\it\small  Shantou, Guangdong 515063, P. R. China}\\
{\it \small email: slzhou@stu.edu.cn}}

\begin{document}
\maketitle
\begin{abstract} In this paper we complete a classification of
finite linear spaces $\cS$ with line size at most 12 admitting a
line-transitive point-imprimitive subgroup of automorphisms. The
examples are the Desarguesian projective planes of orders $4,7, 9$ and $11$,
two designs on 91 points with line size 6, and 467 designs on 729
points with line size 8.

\smallskip\noindent
{\bf Keywords}: linear spaces, line-transitive, point-imprimitive.

\noindent
{\bf AMS subject classification}: 05B05, 
05B25, 
20B25 
\end{abstract}


\section{Introduction}

A {\em finite linear space} 
$\cS=(\cP,\cL)$ consists of a finite set $\cP$ whose elements are
called points, and a set $\cL$ of subsets of $\cP$ whose elements
are called lines, such that each pair of points is contained in exactly
one line and each line contains at least two points. It
is said to be \emph{trivial} if there is only one line, or if all
lines have only two points, and otherwise it is called {\it
non-trivial}. The {\em automorphism group} $\Aut(\cS)$ of $\cS$
consists of all permutations of $\cP$ that leave $\cL$ invariant, and a
subgroup $G$ of $\Aut(\cS)$ is said to be {\em line-transitive} if
it acts transitively on $\cL$. For line-transitive linear spaces the size of lines is
constant, say $k$, and for non-trivial ones $2<k<v$, where
$v=|\cP|$. It is possible for a line-transitive group $G$ to leave
invariant a non-trivial partition of $\cP$, and in this case we say
that $G$ is \emph{point-imprimitive} on $\cS$; otherwise $G$ is
called {\em point-primitive}. In 1989, A. Delandtsheer and J. Doyen
proved that, for a given value of the line size, there are only
finitely many linear spaces that admit a line-transitive,
point-imprimitive group. In 1996, A.R. Camina and S. Mischke
(\cite{CaminaMischke96}) classified all line-transitive,
point-imprimitive linear spaces with line size $k$ at most 8, and recently
such linear spaces with $k/\gcd(k,v)\leq 8$ were classified in \cite{BDLNPZ}. Here we
continue this work, and extend these classifications for line sizes
up to 12. Building on the work in \cite{BDLNPZ,CaminaMischke96}, and the 
results of exhaustive computer searches described in \cite{anton,grids,Cresp,NNOPP},  we
obtain the following classification.

\begin{theorem}\label{theorem:1} \quad Let $\cS$ be a non-trivial finite linear space with $v$ points
and line size $k\leq 12$. Assume that $\cS$ has a group of
automorphisms which is line-transitive and point-imprimitive.
Then $\cS$ is one of the following:
\begin{enumerate}
\item[(a)] A Desarguesian projective plane $\PG_2(q)$ with $q\in
\{4,7,9,11\}$;
\item[(b)] One of two linear spaces with $v=91$ and $k=6$, namely the
Mills design or the Colbourn-McCalla design;
\item[(c)] One of $467$ linear spaces with $v=729$ and $k=8$ constructed by
Nickel, Niemeyer, O'Keefe, Penttila and Praeger in {\rm \cite{NNOPP}},
namely the $N^2OP^2$ designs.
\end{enumerate}
\end{theorem}

In all cases both the full automorphism group $\Aut(\cS)$ and all
line-transitive, point-imprimitive subgroups are known. The main
result that is proved in this paper is Theorem~\ref{theorem:2}. Together with the work 
mentioned above and some non-trivial computer searches discussed in Remark~\ref{rem1},
Theorem~\ref{theorem:2} yields Theorem~\ref{theorem:1}.

\begin{theorem}\label{theorem:2}\quad  Let $\cS=(\cP,\cL)$ be a non-trivial 
finite linear space, admitting a line-transitive, point-imprimitive group
$G$ of automorphisms. Let $v=|\cP|$, and suppose that the line size $k$ 
satisfies $9\leq k\leq 12$ and $\gcd(k,v)=1$. 
Then $G$ has a line regular subgroup $H$, and $k, v, H$
satisfy one of the lines of Table~{\rm\ref{table:k}}. Moreover, in
Lines $1, 3-5$ of Table~{\rm\ref{table:k}}, a point stabilizer in $H$ fixes
a unique point.
\begin{center}
\begin{table}[ht]
$$
\begin{array}{|c|ccc|}
\hline
\mbox{Line}&k & \mbox{$v$}& \mbox{$H$}\\
\hline
1&9 &217& \mbox{$(Z_7\times Z_{31}):Z_3$}\\
2&10&91 & \mbox{$Z_7\times Z_{13}$}\\
3&10&451& \mbox{$(Z_{11}\times Z_{41}):Z_5$}\\
4&11&1431& \mbox{$(Z_3^3\times Z_{53}):Z_{13}$}\\
5&12&133&\mbox{$Z_7\times Z_{19}$}\\
\hline
\end{array}
$$
\caption {Results for Theorem~\ref{theorem:2}\label{table:k}}
\end{table}
\end{center}

\end{theorem}

\begin{remark}\label{rem1}
\begin{enumerate}
\item[(a)] Not all of the lines of Table~{\rm\ref{table:k}} lead to examples of 
line-transitive linear spaces. However this fact cannot easily be demonstrated
theoretically and requires extensive, and exhaustive, computer searches. 
We summarise the results of those searches here.
First, in the case of Line $3$, there are no examples. This fact was 
proved by Greg Cresp {\rm\cite{Cresp}} in his undergraduate honours research project. 
Next the spaces in Line $2$ 
are projective planes, and the fact that the only projective plane admitting a 
line-transitive cyclic group of order $91$
is the Desarguesian plane $\PG_2(9)$ was proved in 
{\rm\cite[Lemma\,  {\rm4.5}]{grids}}. The main aim of {\rm \cite{grids}} 
was to complete a classification initiated in {\rm\cite[Theorem 1.6]{PraegerTuan}}.

Similar but more delicate theory and extensive computation were required to prove
that  for Line $5$, the only  projective plane
admitting a line-transitive cyclic group of order $133$ is $\PG_2(11)$, and 
in Lines $1$ and $4$ there are no examples. 
The theory and 
computational techniques involved in the latter searches are as yet unpublished, 
and will be reported in {\rm\cite{anton}}.
The results of these computer searches, together with Theorem~{\rm\ref{theorem:2}} 
and {\rm\cite[Theorem\, 1]{CaminaMischke96}}, complete the
proof of Theorem~{\rm\ref{theorem:1}}.

\item[(b)] Comparable group theoretic information to that given in Table~{\rm\ref{table:k}}
is available for groups $G$ acting line-transitively and point-imprimitively
on the linear spaces with line size $k\leq 8$ mentioned in Theorem~{\rm\ref{theorem:1}}. 
This is summarised in Table~{\rm\ref{table:groupinfo}}, 
where $F_n$ denotes a  Frobenius group of order $n$, and $P$ is the relatively free, $3$-generator, exponent $3$, nilpotency class $2$ group of order $3^6$. More details about the action of $G$ for the designs in the last line are provided in {\rm\cite[Section 2]{NNOPP}}.

\begin{center}
\begin{table}[ht]
\begin{tabular}{|c|c|c|}
\hline
$\cS$&$G$ &{Reference}\\
\hline
$\PG_2(4)$&$H\times K$, where $H=Z_3$ or $S_3$ and &\cite[Theorem 1.1]{OPP93}\\
&$K=Z_7$ or $F_{21}$&\\
$\PG_2(7)$&$H\times K$, where  $H=Z_3$ and &\cite[p. 50]{atlas}\\
&$K=Z_{19}$ or $F_{57}$&\\
Mills&$G= Z_{91}\cdot Z_d$, where $d\,|\,3$&\cite{StT}\\
Colbourn-McCalla&$G= Z_{91}\cdot Z_d$, where $d\,|\,12$&\cite{StT}\\
NNOPP&$G={\rm Aut}(\cS)=H\cdot Z_{13}$,
where&\cite[Theorem 1.2]{OPP93} and\\
& $H=Z_3^6, Z_9^3$ or $P$
& \cite[Main Theorem 1.3]{NNOPP}\\
\hline
\end{tabular}
\caption {Group information for Remark~\ref{rem1}\label{table:groupinfo}}
\end{table}
\end{center}
\end{enumerate}
\end{remark}

\medskip
One motivation for the research project reported in this paper was to exploit
and test the power of the current linear space theory, as presented in \cite{BDLNPZ}.
The objective of the paper \cite{BDLNPZ} was to collate the available theory 
of finite line-transitive,
point-imprimitive linear spaces, refine it and
develop it further, and then organise it into a series of algorithms that could be 
run on a computer. The output of the algorithms would be a (hopefully manageable)
set of feasible parameters and putative 
line-transitive, point-imprimitive groups for such linear spaces.
These algorithms were implemented, some in C
and others in the {\sf GAP} language \cite{GAP}, and  were applied
in \cite{BDLNPZ} to classify all examples for which $k/\gcd(k,v)\leq 8$. 
In particular this showed that there were no examples with $9\leq k
\leq 12$ and $\gcd(k,v)>1$. 

To
prove Theorem~\ref{theorem:2}, we used the algorithms
in \cite{BDLNPZ} to produce a list of feasible parameters and group 
theoretic information for examples  with $9\leq k
\leq 12$ and $\gcd(k,v)=1$ (see Section 3).
We then made a detailed analysis of each of these possibilities 
in Section 4. Some additional theoretical lemmas used in
Section 4 are presented in Section 2.

\section{Preliminary Results}

\subsection{Notation and hypotheses}

Let $G$ be a transitive group of permutations on a set $\cP$. 
A partition $\cC$ of $\cP$ is \emph{$G$-invariant} if, for all parts 
$C\in\cC$ and all $g\in G$, the image $C^g$ of $C$ under $g$ is 
also a part of $\cC$. For each $G$-invariant partition $\cC$ of $\cP$, 
$G$ induces a transitive permutation group $G^\cC$ on $\cC$ (called the 
\emph{top group}), and the setwise stabiliser $G_C$ of $C\in\cC$ 
induces a transitive permutation group $G^C$ on $C$ (called the 
\emph{bottom group}). Moreover, $G^C$ is independent of the choice 
of  $C$ in $\cC$ up to permutation isomorphism.
The kernel of $G$ on $\cC$ is the subgroup $G_{(\cC)}$ of elements
$g\in G$ with $C^g=C$ for each $C\in \cC$. Thus $G^{\cC}\cong
G/G_{(\cC)}$. We say that $\cC$ is {\it $G$-normal} if $G_{(\cC)}$
is transitive on each of the classes of $\cC$.

For partitions $\cC, \cC'$ of $\cP$, $\cC$ \emph{refines} $\cC'$
if every class of $\cC$ is contained in a class of $\cC'$, and this
refinement is strict if $\cC\ne\cC'$. We also say that $\cC'$ is {\em
coarser, strictly coarser} than $\cC$, respectively. Let $\cC$ 
be a $G$-invariant partition of $\cP$. Then $\cC$ is called 
{\it minimal} if the only strict $G$-invariant refinement
is the discrete partition with all
classes of size 1; and $\cC$ is {\it maximal} if the only 
$G$-invariant partition that is strictly
coarser than $\cC$ is the all-in-one partition (with a single
class). Equivalently, $\cC$ is minimal, maximal if and only if 
$G^C, G^\cC$ is primitive, respectively (where $C\in\cC$).

 Throughout the paper we assume the following
{\sc Hypotheses}. Here we denote by $\Soc(G)$ the \emph{socle} 
of a group $G$, that is, the product of the minimal normal subgroups of $G$.

\medskip
\noindent {\sc Hypotheses.} {\it \quad (a) Let $\cS=(\cP,\cL)$ be a
finite linear space with $v$ points and $b$ lines, each of
size $k$, and with $r$ lines through each point, where $2<k<v$.

(b) Assume that $\cS$ admits a line-transitive, point-imprimitive
subgroup $G$ of  automorphisms which leaves invariant a non-trivial
partition $\cC = \{C_1, \ldots , C_d\}$ of $\cP$ with $d$ classes of
size $c$ where $c>1$ and $d>1$, so that
 $v=cd$ and, by $\cite{DelandtsheerDoyen89}$,
\begin{align}\label{eqn:DD}
c = \frac{\binom{k}{2}-x}{y} \quad \text{and} \quad
d=\frac{\binom{k}{2}-y}{x}
\end{align}
with $x, y$ (called the Delandtsheer-Doyen parameters) positive
integers. \\
Let $\alpha$ denote a point in $\cP$, and let $C\in \cC$ be the class containing  $\alpha$. 

(c) Let the $\cC$-intersection type be $(1^{d_1}, \cdots, k^{d_k})$,
that is,  for each $i$, each line meets $d_i$ classes of $\cC$ 
in $i$ points, and $\sum_iid_i=k$. 

(d) If $\cC$ is $G$-normal, let $K: = G_ {(\cC)}$, $S:
=\Soc(K)$, $X: =C_G(K)$, and $Y:=C_G(S)$.}

\medskip

\subsection{Some lemmas}

First we state some results from Camina and Siemons~\cite{CaminaSiemons} and 
Camina and Praeger~\cite{CaminaPraeger93}.

\begin{lemma}\label{lem:new} Assume that the {\sc Hypotheses} hold.
\begin{enumerate}
\item[(a)] {\rm\cite[Lemma\, 4]{CaminaSiemons}}  Each involution in $G$ fixes at
least one point.
\item[(b)] {\rm\cite[Lemma\,5]{CaminaSiemons}} Suppose that $G$ has a point-regular normal
subgroup $M$. Then no element of $G$ induces the inversion map
$x\mapsto x^{-1}$ on $M$.
\item[(c)] {\rm\cite[Theorem\, 1]{CaminaPraeger93}}
If $N\trianglelefteq G$ then $N$ acts faithfully on each of its
orbits in $\cP$.
\end{enumerate}
\end{lemma}

The next lemma extends \cite[Lemma\,8.2]{BDLNPZ}.

\begin{lemma} \label{lem:KSXY}  Assume that the {\sc Hypotheses} hold and that
$\cC$ is $G$-normal and minimal.
\begin{enumerate}
\item[(a)] Then $\cC$ is the set of $S$-orbits in $\cP$ and
\begin{enumerate}
\item[(i)] Either  $Y\cap K=1$, or $S$ is  elementary abelian and   $Y\cap
K=S$.
\item[(ii)] Either $X\cap K=1$, or $S$ is elementary abelian and $X\cap K=K=S$.
\end{enumerate}
\item[(b)] Suppose in addition that $\cC$ is maximal, that $Y\cap K=S\neq Y$, 
and that one of the following conditions holds.
$$
\begin{array}{|c|c|c|}
\hline
\mbox{Condition} & \mbox{$\Soc(G^{\cC})$}&\mbox{Extra Property}\\
\hline
1& \mbox{abelian} & \mbox{$\gcd(c,d)=1$}\\
2& \mbox{non-abelian} & \mbox{Schur multiplier of a minimal normal subgroup} \\
 &    & \mbox{of $G^\cC$ has no section
isomorphic to $S$}\\
\hline
\end{array}
$$
Then $G$ has a  normal subgroup $M=T\times S$ where $T$ is a
minimal normal subgroup of $G$ and  $T^\cC$ is 
minimal normal in $G^\cC$. Moreover either
\begin{enumerate}
\item[(i)] $T$ is non-abelian and transitive on $\cP$,  or
\item[(ii)] the set $\cC'$ of $T$-orbits in $\cP$ is a $G$-normal
partition of $\cP$ with $|\cC'|=c$ such that for $C\in \cC$ and
$C'\in \cC'$, $|C'|=d$ and $|C\cap C'|=1$. Moreover either $M$ is
regular on $\cP$ or $T$ is not semiregular on $\cP$.
\end{enumerate}

\end{enumerate}
\end{lemma}

\begin{proof} Part (a)  is proved in \cite [Lemma\,8.2]
{BDLNPZ}. For part (b), since $Y\cap K=S$, it follows from part
(a)(i) that $S$ is abelian, and since $\cC$ is minimal the bottom group $G^C$ is primitive. 
By Lemma~\ref{lem:new}(c),
$S^C\cong S$, and since $S^C$ is normal in $G^C$, it follows that 
$S^C$ is regular and is a minimal
normal subgroup of $G^C$. Hence $S$ is a minimal normal subgroup
of $G$. By assumption $Y\nleq K$ and hence $Y^\cC$ is a non-trivial
normal subgroup of $G^\cC$, and since $\cC$ is maximal, $G^\cC$ is primitive. 
Thus there exists a
subgroup $M\vartriangleleft G$ such that $S<M\leq Y$ and $M^\cC\cong
M/S$ is a minimal normal subgroup of $G^\cC$. Suppose first that
condition 1 holds. Then $\Soc(G^\cC)=M^\cC$, $d=p^a$ for some prime
$p$, and a Sylow $p$-subgroup $T$ of $M$ is $Z_p^a$ (since
$\gcd(c,d)=1$). Since $T\leq C_G(S)$, $T$ is 
characteristic in $M$, and hence $T\vartriangleleft G$, and
$M=T\times S$ is regular on $\cP$. By~\cite[Lemma 8.2]{BDLNPZ}, the $T$-orbits form 
a $G$-normal partition of $\cP$ with the properties of (b)(ii).
Since $T^\cC$ is minimal normal in $G^\cC$, it follows that
$T\cong T^\cC$ is minimal normal in $G$, so b(ii) holds.

Now suppose that condition 2 holds. Let $T=M'$, the
derived subgroup of $M$. Since $M^\cC=M/S$
 is a minimal normal subgroup of $G^\cC$ and  $\Soc(G^\cC)$ is non-abelian,
$M/S$ has no non-trivial abelian quotient, and so $T^\cC=M^\cC$. Also, since 
$M\leq Y=C_G(S)$, we have $T\cap S\leq M'\cap Z(M)$, that is, $T\cap S$ 
is contained in the Schur multiplier of $T^\cC$. Moreover, since $S$ is 
minimal normal in $G$ either $T\cap S=1$ or $T\cap S=S$. By condition 2, 
the second possibility cannot occur. So $T\cap S=1$, and $M=T\times S$
with $T$ minimal normal in $G$.

Since $T^\cC$ is transitive, each $T$-orbit $D$ in $\cP$ meets each
$\cC$-class in a constant number of points, say $c_0$, and for
$C\in\cC$, $D\cap C$ is a block of imprimitivity for $G$, and hence
also for $G^C$, of size $c_0$. Since $G^C$ is primitive, $c_0=1$ or
$c$. Thus $T$-orbits in $\cP$ have length $d$ or $v$ respectively.
In the latter case $T$ is transitive, and b(i)
holds. In the former case, by~\cite[Lemma 8.2]{BDLNPZ}, the set of $T$-orbits forms 
a $G$-normal partition of $\cP$ with the properties of (b)(ii). Finally, 
either $M$ is regular on $\cP$ or $T$ is not
semiregular on $\cP$.  
\hfill$\square$
\end{proof}

\begin{lemma}\label{lem:partition} Suppose that the {\sc Hypotheses} hold, 
and also that part (b)(ii) of Lemma~{\rm\ref{lem:KSXY}} holds with $M$ acting
regularly on $\cP$. 
Then we may identify $\cP$ with $M=\{(t,s)|\, t\in T, s\in S\}$ 
such that, for $\alpha=(1_T,1_S)$, $G\leq \hat{G}: =M.({\Aut} (T)\times \Aut (S))$ and $G_\alpha\leq \hat{G}_\alpha=\Aut(T)\times \Aut(S)$ acting naturally on $M$. Moreover,
\begin{align}
\cC&=\{C_t\,|\,t\in T\}\quad \mbox{where}\quad C_t=\{(t,s)\,|\,
s\in S\}\quad \mbox{for $t\in T$}\ \label{eqn:C}\\
\cC'&=\{C_s'\,|\,s\in S\}\quad \mbox{where}\quad C_s'=\{(t,s)\,|\,t\in
T\}\quad \mbox{for $s\in S$}\ \label{eqn:C'}
\end{align}
$K=S.(G_\alpha\cap\Aut(S))$ and the kernel $L$ of the $G$-action 
on $\cC'$ is $L=T.(G_\alpha\cap\Aut(T))$.
\end{lemma}

\begin{proof}
Suppose that $M=T\times S$ is regular on $\cP$, and set $X: =M.\Aut(M)$.
Then we can identify $M$ with $\cP$ such that $G\leq X\leq \Sym(\cP)$, 
and for $\alpha=(1_T,1_S)$, $G=MG_\alpha$ with $G_\alpha\leq X_\alpha=
\Aut(M)$ acting naturally on $M$. Moreover, since $T,
S\vartriangleleft G$ we have $G_\alpha\leq N:=N_{\Aut(M)}(T)\cap
N_{\Aut(M)}(S)$. Now each $\sigma\in N$ induces $\sigma^T\in \Aut
(T)$ and $\sigma^S\in \Aut (S)$ and we have a natural homomorphism
$\varphi : N\rightarrow \Aut(T)\times \Aut(S)$ defined by
$\varphi(\sigma)=(\sigma^T,\sigma^S)$. In fact $\varphi$ is an
isomorphism so we may identify $N$ with $\Aut(T)\times \Aut(S)$.
Setting $\hat{G}=M.N$ we have $\hat{G}_\alpha=N$. Finally, 
since $\cC, \cC'$ are the sets of $S$-orbits and $T$-orbits in
$\cP$, respectively, (2) and (3) hold, and we have $K=S.(G_\alpha\cap\Aut(S))$ and 
$L=T.(G_\alpha\cap\Aut(T))$. \hfill$\square$
\end{proof}

The last lemma of the section addresses a special case of Lemma~\ref{lem:KSXY}~(b)(ii) 
that arises frequently in our search.

\begin{lemma}\label{lem:w} Suppose that the {\sc Hypotheses} hold, 
and also that part (b)(ii) of Lemma~{\rm\ref{lem:KSXY}} holds with $M$ 
regular on $\cP$ and $c$ and $d$ distinct odd primes.
Then $T: =\langle t\rangle\cong Z_d$, $S: =
\langle s\rangle\cong Z_c$, $\Aut (T)= \langle g\rangle\cong
Z_{d-1}$, $\Aut(S)=\langle h\rangle\cong Z_{c-1}$, $G_\alpha\leq
\langle g\rangle \times \langle h\rangle$. Let
$g'=g^{\frac{d-1}{2}}$, $h'= h^{\frac{c-1}{2}}$. Then we may 
re-label the classes of $\cC$ and $\cC'$ so that
\begin{align}
\cC&=\{C_i\, |\, 0\leq i\leq d-1\}\quad{\mbox where}\quad C_i=\{(t^i,
s^j)\, |\, 0\leq j\leq c-1\} \\
\cC'&=\{C_j'\, |\, 0\leq j\leq c-1\}\quad{\mbox where}\quad
C_j'=\{(t^i, s^j)\, |\, 0\leq i\leq d-1 \}
\end{align}
and in particular $C_0=S$ and $C_0'=T$. Moreover, 
\begin{enumerate}
\item[(a)] 
\begin{enumerate}
\item[(i)]
 $\Fix_\cP(g')=C_M(g')=S\in \cC$, $\Fix_\cP(h')=C_M(h')=T\in \cC'$.
 \item[(ii)] $g'h'\not\in G_\alpha$, and $G_\alpha$ contains at most one
of $g'$, $h'$. In particular a Sylow $2$-subgroup of $G$ is cyclic.
 \end{enumerate}
 \item[(b)] If $|G|$ is even then one of (i)--(iv) below holds.
\begin{enumerate}
 \item [(i)] $x\leq 8<y$ and $g'\in G_\alpha$, $\Fix_\cP(g')=S\in
 \cC$, 
 \item [(ii)] $y\leq 8<x$ and $h'\in G_\alpha$, $\Fix_\cP(h')=T\in
 \cC'$, 
 \item[(iii)] $x\geq 9$ and $y\geq 9$,
 \item[(iv)]  $x\leq 8$ and $y\leq 8$, and $\cS$ is $\PG_2(4)$, $\PG_2(9)$, $\PG_2(11)$, 
or the Colbourn-McCalla design.
\end{enumerate}
\end{enumerate}
\end{lemma}

\begin{proof}
The first assertions, as well as (4) and (5), follow from 
Lemma~\ref{lem:partition}. Moreover, as $\hat{G}$ acts naturally on $\cP=M$,
part (a)(i) holds. Since $g'h'$ acts as the inversion map on $M$, 
Lemma~\ref{lem:new}~(b) implies that $g'h'\not\in  G_\alpha$. Thus 
$G_\alpha$ contains at most one of $g'$ and $h'$. Then, since 
$|G:G_\alpha|=cd$ is odd and  
$G_\alpha\leq \langle g\rangle \times \langle h\rangle$,
it follows that $G_\alpha$ contains a Sylow 2-subgroup of $G$ and a
Sylow 2-subgroup is cyclic.

Now suppose that $|G|$ is even. By the previous paragraph, it follows
that  $G_\alpha$ contains exactly one involution $w$, and that $w$
is $g'$ or $h'$. Let $F=\Fix_\cP(w)$. By (a)(i), $F$ is $S$ or $T$ 
respectively. If both $x>8$ and $y>8$ then (b)(iii) holds, so assume that
at least one of $x,y$ is at most 8. If $w=h'$ then $K$ contains
$S.\langle h'\rangle$ and so $K$ is not semiregular on $\cP$. 
Then, by \cite[Theorem 1.6]{PraegerTuan} and \cite[Theorem
1.6]{grids}, either $x>8$ (in which case we must have $y\leq 8$ 
and (b)(ii) holds), or $x\leq8$ and $\cS$ is as in (b)(iv) and for these
linear spaces $y\leq8$ also. Similarly if $w=g'$ then $L$ contains
$T.\langle g'\rangle$ and so $L$ is not semiregular on $\cP$. Again, either
$y>8$ (in which case we must have $x\leq 8$ 
and (b)(i) holds), or $x, y\leq8$ and $\cS$ is as in (b)(iv).\hfill$\square$
\end{proof}

\section{Running the algorithms}

Assume that the {\sc Hypotheses} hold with $9\leq k\leq 12$ and $\gcd(k,v)=1$.
As explained in \cite[Section 3]{BDLNPZ}, to find all pairs $(\cS,G)$ we
may assume that either $G$ is quasiprimitive on $\cP$ (that is to say, all
non-trivial normal subgroups of $G$ are transitive on $\cP$), or $\cC$ is 
$G$-normal. The search procedures in \cite{BDLNPZ} return parameter
and group information for all possibilities $(\cS,G)$ with $k/\gcd(k,v)$
at most a given maximum value $k_{max}^{(r)}$. 
We applied Algorithms 1, 6, 7 and 8 from \cite{BDLNPZ} assuming that
$9\leq k\leq k_{max}^{(r)}=12$ and $\gcd(k,v)=1$. We are
grateful to Dr. Maska Law for performing these computer tests for us. 
The only possibilities for the parameters $v=c\cdot d$, $k, x, y$ and the
intersection type were those given in Table~\ref{Table-fangli}. 
Other parameter information returned by these
algorithms is not shown as it is not needed for our further analysis
in Section 4. The algorithms also returned the group theoretic information
given in Table~\ref{Table-group} about the top group $G^\cC$ and the 
bottom group $G^C$. In particular the computer tests
proved the following:

\begin{enumerate}
\item[(1)] In all {\sc Lines} of Table~\ref{Table-fangli} the partition $\cC$ is $G$-normal (that is, there were no possiblities with $G$ quasiprimitive).
\item[(2)] In all {\sc Lines} of Table~\ref{Table-fangli}, 
$\cC$ is maximal and minimal, so both $G^\cC$ and $G^C$ are primitive.
\end{enumerate}


\begin{center}
\begin{table}[ht]\label{T3}
$$
\begin{array}{|clccccc|}
\hline \mbox{{\sc Line} }& \mbox{$v=d \cdot c$} & \mbox{$x$} &
\mbox{$y$} & \mbox{$k$} & \mbox{$b/v$} & \mbox{{\sc Inter-type}}
\\
\hline\hline
 1 & \mbox{$145= 5 \cdot 29 $} &  $7$& $1$ & $9$ & $2$ & \mbox{$ (1, 2,
 3^{2})$}
\\
2&\mbox{$145= 5 \cdot  29 $} &  $7$& $1$& $9$ & $2$ & \mbox{$
(1^{3}, 2, 4)$}
\\
3 & \mbox{$145= 29 \cdot  5 $} &  $1$& $7$ & $9$ & $2$ & \mbox{$
(1^{7}, 2)$}
\\
4 & \mbox{$217= 7 \cdot  31 $} &  $5$& $1$& $9$ & $3$ & \mbox{$
(1^{2}, 2^{2}, 3)$}
\\
5 & \mbox{$217= 31 \cdot  7 $} &  $1$& $5$& $9$ & $3$ & \mbox{$
(1^{7}, 2)$}
\\
6 & \mbox{$289= 17 \cdot  17 $} &  $2$& $2$ & $9$ & $4$ & \mbox{$
(1^{5}, 2^{2})$}
\\
\hline\hline
 7& \mbox{$91= 7 \cdot  13 $} &  $6$& $3$ & $10$ & $1$  & \mbox{$ (1, 2^{3}, 3)$}
\\
8 &\mbox{$91= 7 \cdot  13 $} &  $6$& $3$  & $10$ & $1$ & \mbox{$
(1^{4}, 3^{2})$}
\\
9 &\mbox{$91= 7 \cdot  13 $} &  $6$& $3$ & $10$ & $1$ & \mbox{$
(1^{6},4)$}
\\
10& \mbox{$91= 13 \cdot  7 $} &  $3$& $6$  & $10$ & $1$ & \mbox{$
(1^{4}, 2^{3})$}
\\
11&\mbox{$91= 13 \cdot  7 $ }&  $3$& $6$ & $10$ & $1$ & \mbox{$
(1^{7}, 3)$}
\\
12 & \mbox{$451= 11 \cdot  41 $} &  $4$& $1$ & $10$ & $5$  & \mbox{$
(1^{2}, 2^{4})$}
\\
13 &\mbox{$451= 11 \cdot  41 $} &  $4$& $1$ & $10$ & $5$ & \mbox{$
(1^{5}, 2, 3)$}
\\
14 & \mbox{$451= 41 \cdot  11 $} &  $1$& $4$ & $10$ & $5$  & \mbox{$
(1^{8}, 2)$}
\\
\hline\hline
 15 & \mbox{$111= 37 \cdot  3 $} &  $1$& $18$ &$11$ & $1$  & \mbox{$ (1^{9}, 2)$}
\\
16 & \mbox{$221= 13 \cdot  17 $} &  $4$& $3$& $11$ & $2$ & \mbox{$
(1^{3}, 2^{4})$}
\\
17 &\mbox{$221= 13 \cdot  17 $} &  $4$& $3$& $11$ & $2$ & \mbox{$
(1^{6}, 2, 3)$}
\\
18 & \mbox{$221= 17 \cdot  13 $} &  $3$& $4$& $11$ & $2$& \mbox{$
(1^{5}, 2^{3})$}
\\
19 &\mbox{$221= 17 \cdot  13 $} &  $3$& $4$& $11$ & $2$& \mbox{$
(1^{8}, 3)$}
\\
20 & \mbox{$441= 9 \cdot  49 $} &  $6$& $1$ & $11$ & $4$& \mbox{$
(1^{2}, 2^{3}, 3)$}
\\
21&\mbox{$441= 9 \cdot  49 $ }&  $6$& $1$& $11$ & $4$ & \mbox{$
(1^{5}, 3^{2})$}
\\
22&\mbox{$441= 9 \cdot  49 $} &  $6$& $1$& $11$ & $4$& \mbox{$
(1^{7}, 4)$}
\\
23 & \mbox{$441= 49 \cdot  9 $} &  $1$& $6$& $11$ & $4$& \mbox{$
(1^{9}, 2)$}
\\
24& \mbox{$1431= 27 \cdot  53 $} &  $2$& $1$& $11$ & $13$ & \mbox{$
(1^{7}, 2^{2})$}
\\
25 &\mbox{$1431= 53 \cdot  27 $} &  $1$& $2$ & $11$ & $13$& \mbox{$
(1^{9}, 2)$}
\\
\hline\hline
26 & \mbox{$133= 7 \cdot  19 $} &  $9$& $3$& $12$ & $1$ & \mbox{$
(2^{3}, 3^{2})$}
\\
27&\mbox{$133= 7 \cdot  19 $} &  $9$& $3$& $12$ & $1$& \mbox{$
(1^{3}, 3^{3})$}
\\
28&\mbox{$133= 7 \cdot  19 $ }&  $9$& $3$& $12$ & $1$ & \mbox{$
(1^{2}, 2^{3}, 4)$}
\\
29&\mbox{$133= 7 \cdot  19 $} &  $9$& $3$& $12$ & $1$ & \mbox{$
(1^{5}, 3, 4)$}
\\
30 &\mbox{$133= 19 \cdot  7 $ }&  $3$& $9$ & $12$ & $1$ & \mbox{$
(1^{6}, 2^{3})$}
\\
31&\mbox{$133= 19 \cdot  7 $} &  $3$& $9$& $12$ & $1$ & \mbox{$
(1^{9}, 3)$}
\\
32& \mbox{$265= 5 \cdot  53 $} &  $13$& $1$& $12$ & $2$& \mbox{$ (2,
3^{2}, 4)$}
\\
33&\mbox{$265= 5 \cdot  53 $} &  $13$& $1$ &$12$ & $2$& \mbox{$
(1^{2}, 2, 4^{2})$}
\\
34&\mbox{$265= 5 \cdot  53 $} &  $13$& $1$ &$12$ & $2$& \mbox{$(1,
2^{3}, 5)$}
\\
35 &\mbox{$265= 53 \cdot  5 $} &  $1$& $13$&$12$ & $2$ & \mbox{$
(1^{10}, 2)$}
\\
36 & \mbox{$793= 13 \cdot  61 $}&  $5$& $1$ & $12$ & $6$&
\mbox{$(1^{2}, 2^{5})$}
\\
37 & \mbox{$793= 13 \cdot  61 $} &  $5$& $1$ & $12$ & $6$&  \mbox{$
(1^{5}, 2^{2}, 3)$}
\\
38& \mbox{$793= 61 \cdot  13 $} &  $1$&$5$ & $12$ & $6$ & \mbox{$
(1^{10}, 2)$}
\\
\hline
\end{array}
$$
\caption{Parameter Sets for $9\leq k\leq 12$}\label{Table-fangli}
\end{table}
\end{center}

\begin{center}
\begin{table}[ht]\label{T4}
$$
\begin{array}{|l|l|l|}
\hline
\mbox{\sc Line} & \mbox{Group $G^{\cC}$}   &  \mbox{Group $G^C$}\\
\hline\hline
1-2  & 5:[2a],\mbox{where $a\mid 2$, or $A_5,S_5$} &   29:[2b],\mbox{where $b=1,2,7$}\\
3  & 29:[2a],\mbox{ where $a\mid 14$}            &   5:[2b],\mbox{where $b\mid 2$}\\
4 & 7:[3a],\mbox{where $a\mid 2$, or $L_3(2)$}  &  31:[3b],\mbox{where $b=1,2,5$}\\
5 & 31:[3a], \mbox{where $a=1,2,5$}             &  7:[3b],\mbox{where $b\mid 2$}\\
6 & 17:[4a],\mbox{where $a\mid 2 $}               &  17:[4b],\mbox{where $b\mid 4$}\\
\hline \hline
7&7:[a], a=1,2,3 & 13:[b], \mbox{where $b=1,2,3,4,6$ }\\
8&7:[a], a=1,2,3 & 13:[b], \mbox{where $b\mid 12$ or $L_3(3)$}\\
9&7:[a],\mbox{$a\mid 6$ or $L_3(2)$, $A_7$, $S_7$} & 13:[b], \mbox{where $b\mid 12$ or $L_3(3)$}\\
10& 13:[a], \mbox{$a=1,2,3,4,6$}& 7:[b], \mbox{where $a\mid 6$}\\
11& 13:[a], \mbox{$a=1,2,3,4,6$} & 7:[b], \mbox{where $b\mid 6$ or $L_3(2)$} \\
12& 11:[5a], \mbox{where $a\mid 2$, or $L_2(11)$}&41:[5b], \mbox{where $b\mid 8$}\\
13& 11:[5a], \mbox{where $a\mid 2$, or $L_2(11)$}&41:[5b], \mbox{where $b\mid 4$}\\
14& 41:[5a], \mbox{where $a\mid 8$} &      11:[5b], \mbox{where $b\mid 2$}\\
\hline \hline
15& 37:[a],\mbox{where $a=4, 12$ or $a\mid 18$} & 3:[b], \mbox{where $b\mid 2$}\\
16& 13:[2a],\mbox{where $a\mid 6$, or $L_3(3)$} & 17:[2b], \mbox{where $b\mid 8$}\\
17& 13:[2a],\mbox{where $a=1,2,3$ }              & 17:[2b], \mbox{where $b\mid 4 $}\\
18& 17:[2a],\mbox{where $a\mid 4$ }              & 13:[2b], \mbox{where $b\mid 6$}\\
19& 17:[2a],\mbox{where $a\mid 8$}              & 13:[2b], \mbox{where $b\mid 6$}\\
20& \mbox{affine or }\mbox{$Soc(G^\cC)=L_2(8)$ or $A_9$}       & \mbox{affine} \\
21-22& \mbox{affine or }\mbox{$Soc(G^\cC)=L_2(8)$ or $A_9$}    & \mbox{affine} \\
23& \mbox{affine or }\mbox{$Soc(G^\cC)=L_3(2)^2$ or  $A_7^2$}   & \mbox{affine} \\
24 & \mbox{affine}                     & 53:[13b],\mbox{where $b\mid 4$}\\
25 & 53:[13a],\mbox{where $a\mid 2$}            & \mbox{affine}\\
\hline \hline
26,28,29 & 7:[a],\mbox{where $a=1,2,3$} & 19:[b], \mbox{where $b\mid 18$ and $b<18$}\\
27 & 7:[a],\mbox{where $a\mid 6$  or $L_3(2)$}& 19:[b], \mbox{where $b\mid 18$}\\
30 & 19:[a],\mbox{where $a\mid 18$ and $a<18$ }& 7:[b], \mbox{where $b\mid 6$}\\
31 & 19:[a],\mbox{where $a\mid 18$ and $a<18$ }& 7:[b], \mbox{where $b\mid 6$}\\
32-34& 5:[2a],\mbox{where $a\mid 2$, or $A_5, S_5$} & 53:[2b], \mbox{where $b=1, 2, 13$}\\
35  & 53:[2a],\mbox{ where $a=1,2, 13$}               & 5:[2b], \mbox{where $b\mid 2$} \\
36 & 13:[6a],\mbox{where $a\mid 2$, or $L_3(3)$}& 61:[6b],\mbox{where $b\mid 10$}\\
37 & 13:[6a],\mbox{where $a\mid 2$, or $L_3(3)$}& 61:[6b],\mbox{where $b=1,2,5$}\\
38 & 61:[6a],\mbox{where $a\mid 10$}           &  13:[6b],\mbox{where $b\mid 2$}\\
 \hline
\end{array}
$$
\caption {Group information for $9\leq k\leq 12$}\label{Table-group}
\end{table}
\end{center}

\section{Detailed Analysis}

Throughout this Section, we assume that the {\sc Hypotheses} hold and that 
$9\leq k\leq 12$. Thus one of the {\sc Lines} in Tables~\ref{Table-fangli} 
and \ref{Table-group} holds.
We give a detailed group theoretic analysis for each of these
{\sc Lines}, thus proving Theorem~\ref{theorem:2}. In the following,  
``{\sc Line} $i$'' will denote the $i$-th line of
Tables~\ref{Table-fangli} and \ref{Table-group}. 
First we deal completely with {\sc Line} 15 and obtain some preliminary 
information for the other {\sc Lines}.

\begin{lemma} \label{lem:1} 
 \begin{enumerate}
 \item[(a)] If $S=Z_c$ and $d\nmid (c-1)$, then $Y\cap K=S\neq Y$ and
 $Y^\cC$ is transitive.
 \item[(b)] {\sc Line} $15$ does not give any examples, and
 moreover one of (i)-(iii) holds.
 \begin{enumerate}
 \item[(i)] One of {\sc Lines} $7-11$ holds, $G$ has a point and line regular cyclic subgroup $Z_7\times Z_{13}$ and $\cS=\PG_2(9)$.
 \item[(ii)] One of {\sc Lines} $26-31$ or $32-35$ holds and $G$ has a
 point-regular normal cyclic subgroup $Z_7\times Z_{19}$ or $Z_5\times Z_{53}$, respectively.
 In {\sc Lines} $26-31$, $\cS$ is a projective
 plane of order $11$.
 \item[(iii)] One of {\sc Lines} $1-6, 12-14, 16-25, 36-38$ holds,
 and $K=S\cong\Soc(G^C)$.
 \end{enumerate}
 \end{enumerate}
 \end{lemma}
\begin{proof}
(a) Recall (from Section 3) that $\cC$ is maximal and minimal.
Suppose that $S=Z_c$ and $d\nmid (c-1)$. Since $G^C$ is primitive on
$C$ with socle $Z_c$, $c$ must be a prime. Hence $G/Y\leq {\rm
Aut}(S)=Z_{c-1}$. So $d\nmid |G/Y|$, and therefore $Y^\cC$ is a
non-trivial normal subgroup of the primitive group $G^\cC$, so
$Y^\cC$ is transitive. In particular $Y\neq Y\cap K$. Also by
Lemma~\ref{lem:new}(c), $K\cong K^C\leq {\rm AGL}(1,c)$, so $Y\cap
K=C_K(S)=S$.

(b) We deal with {\sc Line} 15 in (b)(ii) below.

(b)(i) {\sc Lines 7-11}: Here $\cS$ is a projective plane and in all
these {\sc Lines} the parameter $x$ is 3 or 6. By
Table~\ref{Table-group}, $S={\rm Soc}(K)\cong Z_c$, and $d\nmid
(c-1)$. Thus by part (a), $Y\cap K=S$ and $Y^\cC$ is transitive.
Hence $Y$ contains an element $g$ of order a power of $d$ such that
$\langle g\rangle^\cC$ is transitive. Then $g^d\in Y\cap K=S$ and is
a $d$-element, so $g^d=1$. Thus $\langle g, S\rangle\cong Z_{91}$ is
regular on points and lines of $\cS$. By \cite[Lemma 4.5]{grids},
$\cS$ is the Desarguesian projective plane $\PG_2(9)$.

(b)(ii) {\sc Lines} 15 and 26--35: In each of these {\sc Lines},
$S=Z_c$ is as in Table~\ref{Table-group}, and $d\nmid (c-1)$. So by
part (a), $Y\cap K=S\ne Y$ and $Y^\cC$ is transitive. Hence $Y^\cC$
contains $\Soc(G^{\cC})$ since in all cases $G^\cC$ has a unique 
minimal normal subgroup. In all cases the conditions of
Lemma~\ref{lem:KSXY}(b) hold and so we obtain $T\leq Y$,
$T\vartriangleleft G$ and $T\cong \Soc(G^{\cC})$. Again, in all cases
$c\nmid |T|$ so $T$ is intransitive with $c$ orbits of length $d$,
and the $T$-orbits form another $G$-normal partition $\cC'$ with $c$
classes of size $d$. In particular for {\sc Line} 27, the parameters for 
$\cC'$ should satisfy {\sc Line} 30 or 31, and so $T\cong Z_7$.
Similarly, for {\sc Lines} 32--34, the parameters for 
$\cC'$ should satisfy {\sc Line} 35, and so $T\cong Z_5$. Thus for all
of the {\sc Lines} under consideration, $T\cong Z_d$ and  
$M=T\times S$ is a point-regular normal cyclic
subgroup of $G$. In {\sc Lines} 15 and 26-31, $b/v=1$ so $\cal S$ is a
projective plane. However, for {\sc Line} 15, there is no 
corresponding {\sc Line}  with parameters $(d,c)=(3,37)$, so {\sc Line}  
15 is ruled out. (Note that {\sc Line} 15 corresponds to a projective
plane of order 10, and it was shown in \cite{LTS} that there is no such 
projective plane.)

(b)(iii) In all of these remaining {\sc Lines}, $x\leq 8$, and so by
\cite[Theorem 1.6]{PraegerTuan}, $K=S\cong \Soc(G^C)$.
\hfill$\square$

\end{proof}

\medskip

Note that, by Lemma~\ref{lem:1}, Theorem~\ref{theorem:2} is proved if one of
{\sc Lines} 7--11, 15, or 26--31 holds. Therefore from now on we will assume that one
of the other {\sc Lines} holds, namely {\sc Lines} 1--6, 12--14,
16--25, or 32--38. We next deal with {\sc Lines} 20-23, which required different 
arguments from the other {\sc Lines}.

\begin{lemma}
\label{lem:20-23} {\sc Lines} $20-23$ 
give no examples.
\end{lemma}
\begin{proof} Suppose that one of {\sc Lines} $20-23$ holds and 
let $c=c_0^2, d=d_0^2$, where $c_0=7, d_0=3$ in the case of {\sc Lines} 20-22,
and $c_0=3, d_0=7$ in {\sc Line} 23. By Lemma~\ref{lem:1}(b)(iii) and Table~\ref{T4},
$K=S=\Soc(G^C)\cong Z_{c_0}^2$ and so $G/Y\leq \GL(2, c_0)$. Since
 $S^C\cong S$ is self-centralising in
$G^C$ it follows that $S^C=(Y\cap K)^C\cong Y\cap K$, and hence $Y\cap K=S$. We
claim that $Y^{\cC}\neq 1$. Suppose to the contrary that
$Y^{\cC}=1$. Then $Y=Y\cap K=S$. Thus $Y\leq K$ and hence $|G^{\cC}|$ 
divides $|G/Y|$, which in turn divides $|\GL(2,c_0)|$. In particular $d$ divides
$|\GL(2,c_0)|$, and so $d=9$ and $d$ divides the order of $G/Y\leq\GL(2,7)$. 
Since $27$ does not divide $|\GL(2,7)|$, and since $Y=S=Z_7^2$, it follows that
a Sylow $3$-subgroup $P$ of $G$ is isomorphic to a Sylow $3$-subgroup $Z_3\times Z_3$
of $\GL(2,7)$. Thus $P$ has a subgroup $P_0\cong Z_3$ such that ${P_0Y/Y}
\cong P_0$ is contained in the centre of $\GL(2,7)$. This implies that ${P_0Y/Y}\vartriangleleft G/Y$. 
Hence $P_0S=P_0Y\vartriangleleft G$ and so
$P_0^{\cC}=(P_0S)^{\cC}\vartriangleleft G^\cC$. This is a
contradiction since $|P_0^\cC|=3$ and all nontrivial normal
subgroups of $G^\cC$ are transitive on $\cC$. Hence
 $Y^\cC\ne 1$ and the claim is proved. Now $Y\cap K=S$,
so we have  $S<M\leq Y$ with $M\lhd G$ and $M/S\cong M^\cC\cong
\Soc(G^{\cC})$.

{\sc Case} 1. $M^{\cC}=\Soc(G^\cC)=Z_{d_0}^2$. In this case $M$ is 
abelian with a unique Sylow $c_0$-subgroup
$T=Z_{c_0}^2$ and $M=T\times S$. Moreover $M$ is regular on $\cP$, 
$T$ is normal in $G$ 
and semiregular on $\cP$, and the set of $T$-orbits forms
a second $G$-normal partition $\cC'$ of $\cP$, with $c$ parts of size $d$,
for which one of {\sc Lines} 20-23 holds. 
Thus, for the remainder of our consideration of this case, we may 
assume that $S=Z_7^2$, $T=Z_3^2$, and we may identify $\cP$ with $M$
so that $M$ acts by right multiplication and, for  $\alpha=1_M$, we have $G=M.G_\alpha$ with
$G_\alpha\leq {\rm Aut}(M)=\GL(2,3)\times \GL(2,7)$ acting by
conjugation.

Since $4$ divides $b$, a Sylow 2-subgroup $Q$ of $G$ has order at least 4, and since 
$v$ is odd, $Q$ fixes a point. Replacing $Q$ by a conjugate if necessary, we 
may assume that $Q$ fixes $\alpha$. Let $C\in \cC$, $C'\in\cC'$ such that $C\cap
C'=\{\alpha\}$. Then $Q$ fixes $C,C'$ setwise and we note that $|C'\setminus
\{\alpha\}|=8$. Since $4\mid b$, all $Q$-orbits on $\cL$ have length a 
multiple of 4. In particular $Q$ fixes no lines and
hence $\alpha$ is the unique fixed point of $Q$. Moreover, since for any
point $\beta\ne\alpha$, $Q_{\alpha\beta}$ fixes the line $\la(\alpha,
\beta)$ setwise, it follows that all $Q$-orbits in $\cP\setminus
\{\alpha\}$ have length a multiple of 4. Since all entries in the
$\cC$-intersection type are at most 4, no element of $S$
of order 7 fixes a line, so $S$ is semiregular on $\cL$. Similarly, since
the $\cC'$-intersection type is $(1^9,2)$, no element of $T$ of
order 3 fixes a line, and so $T$ is semiregular on lines. Thus $M$ is
semiregular on $\cL$ with 4 orbits of length $v$. It follows that
$MQ$ is transitive on
 $\cL$, so without loss of generality we may assume that $G= MQ$. Further,
 since each element of order 2 in $Q$ fixes some line,
we have $|Q|\geq 8$.

If $Q\cap K\neq 1$, then $K$ is not semiregular, and
Lemma~\ref{lem:w}(b)(iv) applies giving a contradiction.
Thus $Q\cap K=1$ and so $Q\cong Q^\cC\leq \GL(2,3)$. Each 
Sylow 2-subgroup of $\GL(2,3)$ is isomorphic to $\Gamma{\rm L}(1,9)=Z_8\cdot 2$
and contains the central involution $z_0$ of $\GL(2,3)$. Then, since 
$|Q|\geq8$, it follows that $Q$ contains an involution $z$ such that 
$z^\cC =z_0$, and therefore $\langle M,z\rangle$ is normal in $G$. 
If in addition $z^{\cC'}\neq 1$, then $z^{\cC'}$ is the central involution of
$\GL(2,7)$, and so $z$ inverts each element of $M$, contradicting
Lemma \ref{lem:new}(a). Hence  $z^{\cC'}=1$, and
$G_{(\cC')}$ contains $\langle T,z\rangle$
 which is not semiregular on $\cP$. Again Lemma~\ref{lem:w}(b)(iv) applies and we have a contradiction.

{\sc Case} 2. $M^{\cC}$ is non-abelian.

{\sc Case} 2a. {\sc Lines} 20-22 with $M^{\cC}=L_2(8)$ or $A_9$, or  
{\sc Line} 23 with $M^{\cC}=L_3(2)^2$.

The Schur multiplier of $L_2(8)$, $A_9$ or $L_3(2)^2$ is $1$, $Z_2$,
or $Z_2^2$ respectively, so the conditions of Lemma~\ref{lem:KSXY}(b)
hold. Hence $G$ has a normal
subgroup $T\times S$ where $T\cong M^{\cC}$, and either $T$ is 
transitive on $\cP$ or the set of $T$-orbits forms a $G$-normal partition 
$\cC'$ with $c$ parts of size $d$. In the latter case, parameters for the 
$G$-normal partition $\cC'$ should satisfy {\sc Line} 23 if one of 
{\sc Lines} 20-22 holds for $\cC$, or one of the  {\sc Lines} 20-22
if  {\sc Line} 23 holds for $\cC$. In particular, by Table~\ref{T4}, 
$T$ should be abelian, which is a contradiction. Hence $T$ is
transitive on $\cP$, and in particular $c \mid |T|$. This implies that
$c=9$ and $T=L_3(2)^2$. Then $|T|=2^6 v$ and hence the stabiliser 
$T_\alpha$ in $T$ of a point $\alpha$ is a Sylow $2$-subgroup of $T$.
Thus $T_\alpha = (Z_2^3)^2$ and $N_T(T_\alpha)=(Z_2^3\cdot Z_7)^2$. 
By \cite[Theorem 4.2A]{DM}, the set $F$ of fixed points of 
$T_\alpha$ in $\cP$ has size $|N_T(T_\alpha):T_\alpha|=49= |C_{\rm Sym(\cP)}(T)|$ and 
$C_{\rm Sym(\cP)}(T)$ acts regularly on $F$. However $S=Z_3^2$
centralises $T$ and cannot be a subgroup of a group of order 49.
So this case cannot arise.

{\sc Case} 2b. {\sc Line} 23 with $M^{\cC}=A_7^2$.

Here
$\cC=\cC_1\times \cC_2$, where $\cC=\{C_{ij} | 1\leq i\leq 7,\ 1 \leq
j\leq 7 \}$,  and $A_7^2\leq G^{\cC}\leq S_7 \wr
S_2$ in product action. $\Soc(G^\cC)$ acts on $\cC$ as follows. Let
$(h_1,h_2)\in A_7^2\leq G^\cC$, and $C_{ij}\in \cC$. Then
$C_{ij}^{(h_1,h_2)}=C_{i^{h_1},j^{h_2}}$.
 Let $P$ be a Sylow 5-subgroup of $G$. Since $\gcd(5,b)=1$, $P$ is a
 Sylow 5-subgroup of $G_\la$ for some $\la\in \cL$. Note that if $P$ fixes
 $C_{ij}\in \cC$ setwise, then since $|\AGL(2,3)|$ is not divisible by 5, $P$ must fix
 $C_{ij}$ pointwise. Hence
 $\Fix_\cP(P)$ is a union of $\cC$-classes. Without loss of generality, assume that $P^\cC=\langle
 (12345)\rangle \times \langle
 (12345)\rangle $, so $\Fix_\cC(P)=\{C_{ij}|\,\, i,j=6 \ \ \mbox{or }
 7\}.$
Then $|\Fix_\cC(P)|=4$.
 Since the intersection type is $(1^9,2)$ and since $P\leq G_\la$,  $P$
fixes setwise the unique class, say $C$, with $|C\cap \la|=2$, and $P$
 fixes setwise the 9 classes $C'$ such that $|C'\cap \la|=1$. Since $9\equiv 4\pmod 5$,
$P$ must fix setwise at least 4 of these 9 classes and hence $P$
 fixes at least $1+4=5$ classes setwise. This contradiction
 completes the proof. \hfill$\square$

\end{proof}

\medskip
We now show for all the remaining cases, namely {\sc Lines} 1--6, 
12--14, 16--19, 24--25 and 32--38, that $G$ has a normal abelian point-regular subgroup.

\begin{lemma}\label{lem:2} The group $G$ has a point-regular, normal abelian subgroup $M$,
and either
 \begin{enumerate}
\item[(a)] {\sc Line} $6$ holds, and $M=Z_{17}\times
Z_{17}$ or $Z_{17^2}$, or
\item[(b)] $M=T\times S$, where
\[
S=\left\{\begin{array}{ll}
Z_3^3&\mbox{if {\sc Line} $25$ holds}\\
Z_c  &\mbox{for all other {\sc Lines}}\\
           \end{array}\right.\quad\mbox{and}\quad
T=\left\{\begin{array}{ll}
Z_3^3&\mbox{if {\sc Line} $24$ holds}\\
Z_d  &\mbox{for all other {\sc Lines}.}\\
           \end{array}\right.
\]
\end{enumerate}

\end{lemma}

\begin{proof} For {\sc Lines} 32--35 this is proved in Lemma~\ref{lem:1}~(b).
Suppose therefore that one of the other {\sc Lines} holds. Then, by 
Lemma~\ref{lem:1}~(b)(iii), $K=S=\Soc(G^C)$, and by Table~\ref{T4}, $S$ is
$Z_3^3$ for {\sc Line} $25$, and $Z_c$ with $c$ prime for all the other 
{\sc Lines}. For {\sc Line} $25$, $G/Y\leq \Aut(S)=\GL(3,3)$, and 
so $d=53$ does not divide $|G/Y|$. Hence in this case $Y\not\leq K$, and so 
$Y^\cC\ne1$, which implies that $Y^\cC$ is transitive. For all the other {\sc Lines}
we also have $Y^\cC$ transitive, by Lemma~\ref{lem:1}~(a).

For {\sc Line} 6, since $K=S=Z_{17}$ and $Y^\cC$ is transitive,
there exists a normal subgroup $M$ of $G$ such that $S<M\subseteq Y$
with $M/S=Z_{17}$. So $M=Z_{17}\times Z_{17}$ or $Z_{17^2}$, $M$
is point-regular, and (a) holds. In all the other {\sc Lines}, $\gcd(c,d)=1$, 
and hence the conditions of Lemma~\ref{lem:KSXY}(b) hold  and we obtain
a normal subgroup $M=T\times S$ of $G$, where $T$ is a minimal normal subgroup 
of $G$ and $T^\cC\cong T$ is minimal normal in $G^\cC$.
In all cases $c\nmid |T|$, so $T$ is intransitive on $\cP$ with $c$
orbits of length $d$. The $T$-orbits form a $G$-normal
partition $\cC'$ with $c$ classes of size $d$, and the parameters
for $\cC'$  appear in one of the {\sc Lines} in Table~\ref{T4},
according to the table below.

$$
\begin{array}{|c|c|c|c|c|c|c|c|c|c|}
\hline
\mbox{{\sc Line} for $\cC$} & 1-2& 3& 4 &5 &12-13& 14& 16-17& 18-19& 24\\
\mbox{{\sc Line} for $\cC'$} & 3& 1-2& 5 &4 &14& 12-13& 18-19& 16-17& 25\\
\hline
\end{array}
$$

\medskip\noindent
It follows from Table~\ref{T4} that $T$
must be $Z_d$ for all {\sc Lines} apart from {\sc Line} 24, where $T=Z_3^3$. In all {\sc Lines}
therefore $M=T\times S$ is point-regular, abelian, and normal. 
\hfill$\square$
\end{proof}

\medskip
The next lemma completes our analysis, and thereby completes the proof of Theorem~\ref{theorem:2}.

\begin{lemma}\label{lem:5}
\begin{enumerate}
\item[(a)] If one of {\sc Lines} $4-5, 12-14, 24-25$ holds then one of the {Lines} $1, 3, 4$
of Table~\ref{table:k} holds, respectively.
\item[(b)] {\sc Lines} $1-3, 6, 16-19, 32-38$  give no examples.
\end{enumerate}
\end{lemma}

\begin{proof}
By Lemma~\ref{lem:2}, $G$ has an abelian point-regular normal
subgroup $M$. Moreover in all {\sc Lines} apart from {\sc Line} 6,
$M=T\times S$ with $|T|=d$ and $|S|=c$, while for {\sc Line} 6
either $M$ has this form with $c=d=17$, or $M\cong Z_{17^2}$.
Note that  $|G/M|$ is divisible by $b/v$, where $b/v$ is as in
Table~\ref{Table-fangli}. We will identify $\cP$ with $M$ so 
that $M$ acts by right multiplication, and taking
$\alpha=1_M$, $G=MG_{\alpha}$ with $G_\alpha\leq \Aut(M)$ acting naturally.

First we deal with {\sc Line} 6 in the case
where $M=Z_{17^2}$. It is convenient here to take $M=\cP$ as the
additive group of integers modulo 289, so that $G_\alpha\leq {\rm
Aut}(M)=\{\sigma_{i}\,\, |\,\, \gcd(17,i)=1\}\cong Z_{17}.Z_{16}$, 
where $\sigma_i: a\longmapsto ai$ ($a\in M$). 
Since $b/v=4$ divides $|G_\alpha|$, $G_\alpha$ contains the involution
$\sigma_{-1}$, and this contradicts Lemma~\ref{lem:new}~(b). 

Hence $M=T\times S$ in all cases.
Let $\cC$ be the set of $S$-orbits and $\cC'$ the set of $T$-orbits,
so that $\cC, \cC'$ are $G$-invariant partitions of $\cP$ with
$|\cC|=d$, $|\cC'|=c$. Let $K=G_{(\cC)}\geq S$, and
$L=G_{(\cC')}\geq T$.

Suppose first that one of the {\sc Lines} 1-6, 12-14, 16-19, or
36-38 holds. Then, by Lemma~\ref{lem:1}, $K=S\cong Z_c$ and
$L=T\cong Z_d$. Also by Lemma~\ref{lem:w}, we may set $T=\langle
t\rangle\cong Z_d$, $S=\langle s\rangle\cong Z_c$, $\cP=M$,
$\alpha=(1,1)$, $\Aut(T)=\langle g\rangle\cong Z_{d-1}$,
$\Aut(S)=\langle h\rangle\cong Z_{c-1}$ and $G_\alpha\leq \langle
g\rangle \times \langle h\rangle$. Since $K=S$ and $L=T$ we have
$G_\alpha\cap \langle g\rangle=  G_\alpha\cap \langle h\rangle=1$.
Thus $G_\alpha$ is cyclic of order dividing $\gcd(c-1,d-1)$, and  $b/v$ divides
$|G/M|=|G_\alpha|$. In all of these {\sc Lines}, both $x\leq8$ and $y\leq8$,
and it follows from Lemma~\ref{lem:w} (b) that $|G|$ is odd, and
in particular $b/v$ is odd. The only possibilities are {\sc Lines} 4-5 with 
$G_\alpha\cong Z_3$, and {\sc Lines} 12-14 with 
$G_\alpha\cong Z_5$. In both cases $G$ is line-regular, and so Line 1 or 3 
of Table~\ref{table:k} holds, respectively.

Now suppose that  one of the {\sc Lines} 24-25 holds. By Lemma~\ref{lem:1},
$K=S$ and $L=T$, and without loss of generality we may take $T=Z_3^3$,
and $S=Z_{53}$. By Lemma~\ref{lem:partition}, $G_\alpha\leq \Aut(T)\times
\Aut(S)=\GL(3,3)\times Z_{52}$ with $G_\alpha\cap
\Aut(T)=G_\alpha\cap \Aut(S)=1$. In particular $G_\alpha$ is cyclic
of order dividing 52 and divisible by $b/v=13$. Thus $G$ has a
line-regular subgroup $H=M.Z_{13}$ with $H_\alpha=Z_{13}$ fixing
only the point $\alpha$, and Line 4 of Table~\ref{table:k} holds.

Finally suppose that one of the {\sc Lines} 32-35 holds. Without 
loss of generality we may take  $T= Z_5$ and $S=Z_{53}$. 
Now $b/\gcd(b,v)=2$ divides $|G_\alpha|$, so $|G|$ is even. By
Lemma~\ref{lem:w}(b)(ii), $G_\alpha$ contains a unique involution
$w$ that inverts $S$ and centralises $T$. 
Now $w=h^{26}$ where $\langle h\rangle ={\rm Aut}(S)$,
and we set  $\langle g\rangle
={\rm Aut}(T)$.
Since $b=2v$, we have
$|G_\alpha|=2|G_\la|$, and since each involution fixes a line it
follows that $|G_\la|$ is even and so $4\mid |G_\alpha|$.
It follows from Lemma~\ref{lem:w}(b) that a
Sylow 2-subgroup of $G_\alpha$ is cyclic so there is an element
$u\in G_\alpha$ of order 4 such that $u^2=w$. The only
possibilities for $u$ are $h^{13}$ and $g^2h^{13}$ or their inverses
$h^{39}$, $g^2h^{39}$ respectively. Thus we may assume that $u=h^{13}$ or
$g^2h^{13}$. In either case $u$ fixes setwise the set $T=\Fix_\cP(w)
\in\cC'$. If $u=h^{13}$ then $u$ fixes $T$ pointwise, while if $u=g^2h^{13}$,
then $u$ fixes $\alpha$ and has two orbits of length $2$ in $T$. In either case
$u$ fixes setwise a pair $\{\beta,\gamma\}$ of points of $T$. Let
$\la'=\la(\beta,\gamma)$ be the line through $\beta$ and
$\gamma$. Since the $\cC'$-intersection type is $(1^{10},2)$ it
follows that $|\la'\cap T|=2$, and that $u$ fixes setwise the 10 
classes $C'\in\cC'$ such that $|C'\cap\la'|=1$. This is a contradiction since
$\langle u\rangle$ fixes the class $T\in\cC'$ and permutes the remaining 
52 classes in orbits of length 4. \hfill$\square$
\end{proof}

\bigskip
Theorem~\ref{theorem:2}
follows from Lemma~\ref{lem:1}-~\ref{lem:5}. 

\bigskip
\noindent{\bf\large Acknowledgments}

The authors are very grateful to Maska Law for running the programs
from \cite{BDLNPZ} for line size $k\leq 12$ and some useful discussion.
They also thank Anton Betten for conducting the searches on the
possibilities in Table~\ref{table:k}.

\end{document}